\documentclass[a4paper,12pt]{article}
\usepackage{amssymb}
\usepackage{amsfonts}
\usepackage{amsmath}
\usepackage{amsthm}

\def\es{\emptyset}
\def\eps{\varepsilon}

\def\Z{\mathbb{Z}}

\def\extr{\mathrm{ext}}

\def\cD{\mathcal{D}}
\def\cE{\mathcal{E}}

\def\cR{\mathcal{R}}

\def\Pr{\mathbf{P}}
\def\E{\mathbf{E}}

\newcommand{\eqn}[2]{\begin{equation}\label{#1}#2\end{equation}}
\newcommand{\eqnst}[1]{\begin{equation*}#1\end{equation*}}

\newcommand{\eqnsplst}[1]{\begin{equation*}\begin{split}%
    #1\end{split}\end{equation*}}

\theoremstyle{plain}
\newtheorem{theorem}{Theorem}
\newtheorem{lemma}[theorem]{Lemma}

\theoremstyle{definition}
\newtheorem{definition}{Definition}

\theoremstyle{remark}
\newtheorem{remark}{Remark}

\newtheorem{open}{Open Question}

\begin{document}

\title{Minimal configurations and sandpile measures}
\author{Antal A.~J\'arai\thanks{Department of Mathematical 
Sciences, University of Bath,
Claverton Down, Bath BA1 7AY, United Kingdom. 
E-mail: {\tt A.Jarai@bath.ac.uk}} \ and Nicol\'as Werning
\thanks{Department of Mathematics and Statistics, 
University of Reading, Whiteknights, Reading RG6 6AX, 
United Kingdom. 
E-mail: {\tt N.Werning@pgr.reading.ac.uk}}}

\maketitle

\abstract{We give a new simple construction of the 
sandpile measure on an infinite graph $G$, under the sole
assumption that each tree in the Wired Uniform Spanning
Forest on $G$ has one end almost surely. 
For, the so called, generalized 
minimal configurations the limiting probability on $G$ 
exists even without this assumption. We also give
determinantal formulas for minimal configurations
on general graphs in terms of the transfer current 
matrix.}

\medbreak


\section{Introduction}
\label{sec:intro}

In this paper we study minimal configurations and associated 
determinantal formulas in Abelian sandpiles. This will lead 
to a new simple construction of sandpile measures.
Let us start by defining the Abelian sandpile model, 
deferring more detailed background to Section \ref{sec:defn}.

Let $G = (V \cup \{ s \}, E)$ be a finite connected
multigraph, with a distinguished vertex $s$,
called the ``sink''. 
A \emph{sandpile} on $G$ is a configuration of 
particles on $V$, specified by a map 
$\eta : V \to \{ 0, 1, 2, \dots \}$, where 
$\eta(x)$ is the number of particles at $x$.
If $\eta(x) \ge \deg_G(x)$, the vertex $x$ can 
\emph{topple} and send one 
particle along each edge incident with $x$.
Particles that reach the sink are lost. A sandpile 
is \emph{stable}, if no vertex can topple, that is,
if $\eta(x) < \deg_G(x)$ for all $x \in V$.

We define a Markov chain on the set of stable 
sandpiles as follows. At each step, we add a particle
at a uniformly random vertex of $V$, and carry out 
any possible topplings until a stable sandpile is
reached. It was shown by Dhar \cite{Dhar90} that 
the resulting stable sandpile does not depend
on the order of topplings (Abelian property), and 
the stationary distribution is unique and uniformly 
distributed on the set of recurrent states. 
We denote the stationary distribution by $\nu_G$
or $\nu_V$. Bak, Tang and Wiesenfeld \cite{BTW87} 
introduced the above model in a less general setting,
prior to the work of Dhar, and hence the model 
is also known as the BTW sandpile. See the 
surveys \cite{Redig} and \cite{HLMPPW} for background.

When $G = (V,E)$ is an infinite, locally finite,
connected graph, 
and $V_1 \subset V$ is a finite set, we form
the graph $G_{V_1} = (V_1 \cup \{ s \}, E_{V_1})$,
by identifying all vertices in $V \setminus V_1$
to a single vertex, that becomes the sink $s$
of $G_{V_1}$, and removing loop-edges at $s$. 
It was shown by Athreya and J\'arai \cite{AJ04}
that when $G = \Z^d$, $d \ge 2$, and 
$V_1 \subset V_2 \subset \dots \subset \Z^d$ is a sequence
of cubes with union $\Z^d$, then the stationary 
measures $\nu_{V_n}$ converge weakly to a limit $\nu$,
called the \emph{sandpile measure} of $\Z^d$.
This was generalized to certain other infinite graphs 
in \cite{Jarai11}. Our main result in this paper will be 
a new simple construction of the sandpile measure $\nu$
on general graphs $G$, under the sole assumption:
\eqn{e:one-end} 
{ \parbox{7cm}{each tree in the Wired Uniform Spanning Forest on 
    $G$ has one end almost surely.} } 
The Wired Uniform Spanning Forest is a random spanning subgraph
of $G$ that is obtained, through a limiting process,
from uniformly random spanning trees of finite graphs;
see Section \ref{sec:defn}, where we also define the notion 
of ``end''.
For certain special cylinder events, called generalized minimal 
subconfigurations, the limiting probability on $G$ will be
shown to exist even without assumption \eqref{e:one-end}. 

A result of Majumdar and Dhar \cite{MD92} plays an
important role in our proofs. These authors
used the burning algorithm of Dhar \cite{Dhar90} to 
construct a bijection between recurrent sandpiles on $G$ 
and spanning trees of $G$. Since the stationary measure 
$\nu_G$ is uniform on the set of recurrent sandpiles,
the burning bijection maps it to the uniform measure
on spanning trees of $G$. This is known as the Uniform 
Spanning Tree measure; see e.g.~\cite{LPbook} for 
background.

In many ways, the Uniform Spanning Tree is 
an easier object than recurrent sandpiles.
One of its features that we use in this paper is that
its marginal on a fixed set of edges is given by a 
simple determinantal formula. Let $T_G$ denote a random spanning 
tree of $G$, chosen according to the uniform distribution. 
The Transfer Current Theorem of 
Burton and Pemantle \cite{BP93} implies that
the edges of $T_G$ form a \emph{determinantal process}.
That is, there exists a matrix $Y_G(e,f)$, $e, f \in E$, the 
\emph{transfer current matrix}, such that for any 
$k \ge 1$ and distinct edges $e_1, \dots, e_k \in E$ 
we have
\eqn{e:TCT}
{ \Pr [ e_1, \dots, e_k \in T_G ]
  = \det ( Y_G(e_i, e_j) )_{i,j=1}^k. }
The matrix $Y_G$ arises from a connection between spanning
trees, electrical networks and random walk; see for example
\cite{BLPS01} for an exposition of these connections.
For the sake of this introduction, we state the interpretation 
of $Y_G$ in terms of random walk. Orient each edge in $E$ 
arbitrarily. Given oriented edges $e$ and $f$, consider 
simple random walk on $G$ started at the tail of $e$
and stopped at the first time it reaches the head of $e$.
Let $J^e(f)$ denote the expected net usage of $f$, that 
is, the expected number of times the walk uses $f$, minus the
expected number of times it uses the reversal of $f$.
Then $Y_G(e,f) = J^e(f)$ can be taken as the definition 
of $Y_G$; see \cite{DS84} or \cite[Theorem 4.1]{BLPS01}. 
Note that it is not difficult to see from this definition, 
using reversibility of the random walk, that the determinant 
on the right hand side of \eqref{e:TCT} does not depend
on the chosen orientation of the edges.
As pointed out in \cite{BP93}, the transfer current matrix 
can be expressed in terms of the 
Green function of simple random walk on $G$. 
In order to state this, first note that $J^e(f)$ is
not affected if we replace the discrete time
simple random walk by the continuous time simple 
random walk $\{ S(t) \}_{t \ge 0}$ that crosses 
each edge at rate $1$. (The generator
of $S$ is the negative of the graph Laplacian.)
For vertices $x, y \in V \cup \{s \}$, let 
\eqnst
{ H(x,y)
  := \lim_{t \to \infty} 
    \E \Big[ \int_0^t \left(I [ S(t) = y ] 
    - I [ S(t) = x ] \right) dt 
    \Big| S(0) = x \Big]. }
The limit exists due to exponentially fast convergence 
to the uniform stationary distribution.
If the tail and head of $e$ are $x = \underline{e}$ 
and $y = \overline{e}$, and the tail and head of $f$ are
$u = \underline{f}$, $w = \overline{f}$, then 
it is not difficult to see that 
$J^e(f) = H(x,u) - H(y,u) - H(x,w) + H(y,w)$.

No simple expression similar to \eqref{e:TCT} is
known for the marginal of $\nu_G$ on a fixed
set of vertices. On the other hand, some determinantal 
formulas exist for special subconfigurations. 
Majumdar and Dhar \cite{MD91} showed that on $\Z^d$, 
$d \ge 2$, the probability 
$p_0(d) = \nu [ \eta : \eta(0) = 0 ]$ 
can be written as a determinant involving the
simple random walk potential kernel $a(x) = H(0,x)$ ($d = 2$) 
or the Green function $G(x)$ ($d \ge 3$); 
see e.g.~\cite[Chapter 4]{LL10} for the definitions of
$a(x)$ and $G(x)$. 
In $d = 2$ the result is the explicit value 
$p_0(2) = \frac{2}{\pi^2} - \frac{4}{\pi^3}$. 
Majumdar and Dhar also showed that in dimensions 
$d \ge 2$, the correlation between
the events of seeing no particle at $x$ and $y$,
respectively, decays as
\eqnst
{ \nu [ \eta(x) = 0, \eta(y) = 0 ] - p_0(d)^2
  \sim c |x - y|^{-2d}, \quad\text{as $|x - y| \to \infty$.} } 
More generally, the probability of the event that none of the 
vertices $x_1, \dots, x_k$ has a particle, is given by a 
``block-determinantal'' formula \cite{MD91}. 
This can be written as
\eqnst
{ \nu [ \eta(x_1) = 0, \dots, \eta(x_k) = 0 ]
  = \det ( M(i,j) )_{i,j=1}^k. }
where, in its most reduced form, each $M(i,j)$ is 
a $(2d-1) \times (2d-1)$ matrix block. The above explicit
form was exploited by D\"{u}rre \cite{Duer09}, 
who gave rigorous scaling limit results in 2D for the 
random field of vertices having no particles.

In its most general form, the method of Majumdar and Dhar
applies to \emph{minimal subconfigurations}. We say that
a stable configuration $\xi$ of particles on a subset 
$W \subset V$ is \emph{minimal}, if it has an extension
to a recurrent sandpile on $V$, but removing a particle 
from any of the vertices in $W$ would render such an
extension impossible. (In fact, for technical reasons, 
we are going to distinguish between \emph{minimal} and
\emph{generalized minimal} subconfigurations, but this
distinction can be ignored for now.) 
The computations quoted above are all examples of the form
$W = \{ x_1, \dots, x_k \}$, $\xi(x_i) = 0$, $i = 1, \dots, k$.
In Theorems \ref{thm:minimal-finite} and 
\ref{thm:minimal-infinite} below, we formulate the
general statement that the probability of any minimal 
subconfiguration can be written as a determinant involving 
the transfer current matrix. This type of result appears
to be taken for granted in the physics literature,
see e.g.~\cite{MR01}. Yet, we have not found it clearly 
stated anywhere in a general form, and it seems worthwhile
to record it here. 
Our formulation in terms of the transfer current matrix
is slightly different from what is usually used in the
physics community. We believe that our formulation 
highlights what makes the theorem work: namely that 
minimal events can be expressed, via the burning bijection, 
as the absence of a fixed set of edges from the Uniform 
Spanning Tree.

Consider now the probabilities of the non-minimal events:
\eqnst
{ p_i(d) 
  := \nu [ \eta : \eta(0) = i ], \quad i = 1, \dots, 2d-1, 
     \quad d \ge 2, } 
where $\nu$ is the sandpile measure on $\Z^d$.
Majumdar and Dhar \cite[Eqn.~(14)]{MD91}
gave an infinite series for the value of $p_1(2)$
where each term in the series can be written as a
determinant. A rigorous proof that the series indeed
converges to $p_1(2)$ can be given based on the fact that 
assumption \eqref{e:one-end} is satisfied for $\Z^d$,
$d \ge 2$. A similar series can be given for $p_i(d)$
in general. It was by considering extensions of 
the series of Majumdar and Dhar that we arrived at our 
main result, Theorem \ref{thm:sandpile}, saying that 
under assumption \eqref{e:one-end}, a unique sandpile measure
$\nu$ exists. Some of the arguments of Levine and Peres 
\cite{LP11} were also inspiring, who prove fascinating
connections between the average number of particles
$\sum_{i = 0}^{d-1} i p_i(d)$, and seemingly unrelated
constants in other models. 

Construction of the sandpile measure $\nu$ is the first 
step in understanding the asymptotic behaviour of the model
on a growing sequence of subgraphs of an infinite graph $G$. 
Hitherto $\nu$ has only been shown to exist under more restrictive
assumptions on $G$, such as connectedness of the Wired
Uniform Spanning Forest, or for certain transitive graphs; 
see \cite{AJ04,Jarai11}.
We believe that the greater generality of our theorem 
will be useful in studying Abelian sandpiles on some irregular 
graphs, for example, graphs obtained as the result of a 
random process. We know that assumption \eqref{e:one-end} 
cannot be omitted: J\'arai and Lyons \cite{JL07} show that 
on graphs of the form $G = \Z \times G_0$, where $G_0$ is
any connected finite graph with at least two vertices,
there are two distinct ergodic weak limit points of the 
family $\{ \nu_U : \text{$U \subset V(G)$, $U$ finite} \}$.
The following question, that is a strengthened form of 
a question of \cite{JL07}, remains open, even in the case 
of $\Z^2$. 

\begin{open} 
\label{open:avalanches}
Assume that $G$ is recurrent, satisfies 
assumption \eqref{e:one-end}, and $o$ is a fixed vertex 
of $G$. Draw a configuration from the measure $\nu$, add
a particle at $o$, and carry out all possible topplings.
Is it true that there are finitely many topplings $\nu$-a.s.?
\end{open}

Note that the statement holds when $G$ is transient and 
satisfies assumption \eqref{e:one-end}; this follows by the
argument of \cite[Theorem 3.11]{JR08}.

In Section \ref{sec:defn} we give further definitions and 
state our results. In Section \ref{sec:prelims}
we collect preliminary results. In Section \ref{sec:sandpile} 
we give the general construction of sandpile measures.

A brief announcement of our results appeared in the 
proceedings \cite{OWR}.

\section{Definitions and main results}
\label{sec:defn}

Let $G = (V \cup \{ s \}, E)$ be a finite connected 
multigraph.  We write $b_{uv}$ for the number of edges 
between vertices $u$ and $v$, and write $u \sim v$ if $b_{uv} \ge 1$. 
We write $\cR_G$ for the set of recurrent sandpiles on $G$.
Let $F \subset V$. We say that the stable sandpile $\eta$ 
on $G$ is \emph{ample} for $F \subset V$, if there exists 
$x \in F$ such that $\eta(x) \ge \deg_F(x)$, where 
$\deg_F(x) = \# \{ y \in F : x \sim y \}$. (Here $\# A$
denotes the number of elements of $A$.) By the well-known 
burning test of Dhar \cite{Dhar90} (see also 
\cite[Lemma 4.2]{HLMPPW}), we have
\eqn{e:ample}
{ \text{$\eta \in \cR_G$} \qquad \text{if and only if} \qquad 
    \text{$\eta$ is ample for every nonempty $F \subset V$.} }
Let $W \subset V$. We define the graph $G_W$ by ``wiring the 
complement of $W$'', i.e.~identifying
all vertices in $V \setminus W$ with $s$, and removing 
loop-edges. Due to the criterion \eqref{e:ample}, the restriction
of any $\eta \in \cR_G$ to $W$, denoted $\eta_W$, is in $\cR_{G_W}$. 
When the choice of $G$ is clear from the context, 
we write $\cR_W$ for $\cR_{G_W}$. 

The burning bijection \cite{MD92}
establishes a one-to-one mapping between recurrent
sandpiles and spanning trees of $G$. We will need a 
particular version of this bijection that is 
introduced in Section \ref{ssec:bijection}. 

There is a relationship between spanning trees
and electrical networks, for which we refer the
reader \cite{LPbook} or \cite[Section 4]{BLPS01}. 
The \emph{transfer current matrix} $Y_G$ is defined
by regarding $G$ as an electrical network, as follows.
Choose an orientation for each edge of $G$. Replace each edge
of $G$ by a wire of unit resistance, and hook up a
battery between the two endpoints of $e$. Suppose
that the voltage of the battery is such that in the network
as a whole, unit current flows from the tail of $e$ to the
head of $e$. Let $I^e(f)$ be the amount of current flowing 
through edge $f$ consistent with its orientation. (Hence 
$I^e(f)$ can be positive or negative depending on whether the
current flows in the same direction or not as the orientation
of $f$). The matrix $Y_G(e,f) := I^e(f)$, indexed by the
edges of $G$, is the \emph{transfer current matrix}.
(See \cite{DS84} for a proof that $I^e(f) = J^e(f)$.)
An example is given in Figure \ref{fig:current-example}.
\begin{figure}
\centering
\setlength{\unitlength}{1cm}
\begin{picture}(15,6)
\thicklines
\put(1,2){\circle*{0.2}}
\put(5,2){\circle*{0.2}}
\put(3,5){\circle*{0.2}}
\qbezier(4.5,1.8)(3,1)(1.5,1.8)
\put(1.5,1.8){\vector(-3,1){0}}
\qbezier(4.5,2.2)(3,3)(1.5,2.2)
\put(1.5,2.2){\vector(-3,-1){0}}
\qbezier(1.2,2.5)(2,3.5)(2.8,4.5)
\put(2.8,4.5){\vector(3,4){0}}
\qbezier(4.8,2.5)(4,3.5)(3.2,4.5)
\put(4.8,2.5){\vector(3,-4){0}}
\put(2.9,1){$e_1$}
\put(2.9,2.8){$e_2$}
\put(1.6,3.6){$f$}
\put(4.4,3.6){$g$}
\put(7,3){$\displaystyle{Y_G = 
\begin{pmatrix}  
2/5 & 2/5 & -1/5 & -1/5 \\
2/5 & 2/5 & -1/5 & -1/5 \\
-1/5 & -1/5 & 3/5 & -2/5 \\
-1/5 & -1/5 & -2/5 & 3/5
\end{pmatrix}}$}
\end{picture}
\caption{Example of a transfer current matrix. The columns 
correspond to the edges in the order $e_1, e_2, f, g$.
The entries can be computed by simple applications of the 
series-parallel laws. There are $5$ spanning trees.}
\label{fig:current-example}
\end{figure}
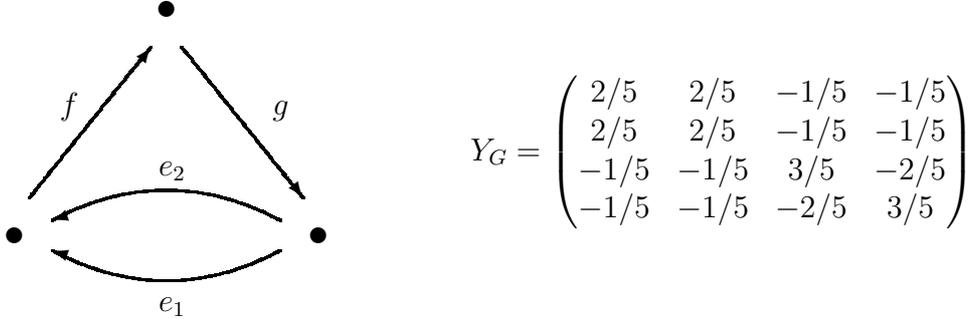

We define the matrix
\eqnst
{ K_G(e,f)
  := \delta_G(e,f) - Y_G(e,f), }
where $\delta_G$ is the identity matrix. An extension of the 
Transfer Current Theorem \cite[Corollary 4.4]{BP93}
implies that if $e_1, \dots, e_k$ are distinct
edges and $T_G$ is a uniformly random spanning tree of $G$,
then
\eqn{e:extendedTCT}
{ \Pr [ e_1, \dots, e_k \not\in T_G ]
  = \det ( K_G(e_i,e_j) )_{i,j=1}^k. }

We always regard $T_G$ as a rooted tree, with root at $s$.
By analogy with family trees, we call a vertex $x$ 
a \emph{descendent} of the vertex $y$, if $y$ lies on the 
unique path from $x$ to $s$ (here we allow $x = y$).

The above can be extended to a locally finite, connected, infinite 
graph $G = (V,E)$ as follows. By an exhaustion of $G$ we mean 
a sequence of finite subsets 
$V_1 \subset V_2 \subset \dots \subset V$ such that
$\cup_{n \ge 1} V_n = V$. Form the finite graphs
$G_n := G_{V_n} = (V_n \cup \{ s \}, E_n)$ by identifying 
$V \setminus V_n$ to a single vertex $s$, and removing 
loops at $s$. Note that each edge in $E_n$ can be regarded,
in a natural way, as an element of $E$.
Let $\mu_n$ denote the probability measure on 
$\{ 0, 1 \}^{E_n}$ that is supported on spanning trees of $G_n$ 
and gives equal weight to each spanning tree.
Here the value $1$ corresponds to an edge being present,
and we identify any spanning subgraph of $G_n$ with the 
set of edges it contains. We write $T_{G_n}$ for the random 
variable on $\{ 0, 1 \}^{E_n}$ that is the collection 
of edges contained in the corresponding spanning subgraph.
We similarly define the random variable $T_G$ on the 
space $\{ 0, 1 \}^E$.
By a result of Pemantle \cite{Pem91} (see also \cite{BLPS01}),
the weak limit $\mu = \lim_{n \to \infty} \mu_n$ exists,
as a measure on $\{ 0, 1 \}^E$, and is independent of the 
exhaustion. That is, for any finite
sets $B, K \subset E$ we have
\eqnst
{ \mu [ T_G \cap K = B ]
  = \lim_{n \to \infty} \mu_n [ T_{G_n} \cap K = B ]. }
The measure $\mu$ is called the 
\emph{Wired Uniform Spanning Forest measure} of $G$
(we henceforth abreviate this to WSF).
The term ``wired'' refers to the particular choice 
of boundary condition, that is, the identification
of vertices in $V \setminus V_n$. 
The measure $\mu$ is supported on spanning subgraphs 
of $G$, each of whose components is an infinite tree.

A \emph{ray} in an infinite tree is an infinite self-avoiding path.
An \emph{end} of an infinite tree is defined as an equivalence 
class of rays, where two rays are equivalent, if they have a finite 
symmetric difference. Hence an infinite tree has one end 
if and only if it contains no two disjoint infinite
self-avoiding paths.

The \emph{wired current} in $G$ is the pointwise limit 
$I^e = \lim_{n \to \infty} I^e_{V_n}$ that exists by monotonicity; 
see e.g.~\cite{BLPS01}.
Both \eqref{e:TCT} and \eqref{e:extendedTCT} have 
limiting versions on $G$ \cite[Theorem 4.2]{BP93},
involving $Y_G(e,f) = I^e(f) = \lim_{n \to \infty} Y_{G_n}(e,f)$
and $K_G(e,f) = \lim_{n \to \infty} K_{G_n}(e,f)$.

Properties of the WSF have been studied 
extensively. Pemantle \cite{Pem91} proved that on $\Z^d$,
$2 \le d \le 4$, $T_G$ is a single tree and has one end
$\mu$-a.s. He also proved that when $d \ge 5$, $T_G$ 
consists of infinitely many trees, and each has one or 
two ends $\mu$-a.s. This was completed and extended 
by Benjamini, Lyons, Peres and Schramm \cite{BLPS01},
who showed in particular that in the case $d \ge 5$
each tree has one end $\mu$-a.s. In fact, their general
result (Theorem 10.1 in the reference above) implies that 
on any Cayley graph that is not a finite extension of 
$\Z$, each tree in $T_G$ has one end $\mu$-a.s. 
Examples of graphs where $T_G$ is connected and 
has two ends $\mu$-a.s.~are provided by graphs of
the form $G = \Z \times G_0$, where $G_0$ is a finite,
connected graph; see \cite[Proposition 10.10]{BLPS01}
for a more general result.
Lyons, Morris and Schramm \cite{LMS08} gave a very general
condition on the isoperimetric profile of a graph
that ensures that each tree in $T_G$ has one end $\mu$-a.s. 
From the above it is clear that assumption \eqref{e:one-end} 
is known to hold on many graphs.

We regard any one-ended infinite tree as being ``rooted at
infinity'', and we call vertex 
$x$ a \emph{descendent} of vertex $y$, if
$y$ lies on the (necessarily unique) infinite 
self-avoiding path in the tree that starts at $x$.
(Here we allow $x = y$.)

The starting point for this paper is the notion of a minimal 
configuration. It is well-known, and easy to see 
using \eqref{e:ample}, that if $\eta$ is a recurrent sandpile 
on $G$, then adding more particles to $\eta$ also results 
in a recurrent sandpile, as long as it remains stable. 
Let $\delta_u$ denote the sandpile with a single 
particle at $u$ and no other particles. 

\begin{definition}
\label{defn:minimal}
Let $G = (V \cup \{ s \}, E)$ be a finite, 
connected multigraph and let
$\es \not= W \subset V$ be such that $G \setminus W$ (the graph 
obtained from $G$ by removing the vertices in $W$) is connected. 
Let $\xi$ be a stable configuration on $W$. 
We say that $\xi$ is \emph{minimal}, if there exists a recurrent
sandpile $\eta \in \cR_G$ such that $\eta_W = \xi$, and for 
any such $\eta$ and any $w \in W$, we have
$\eta - \delta_w \not\in \cR_G$.
\end{definition}

\begin{remark}
Equivalently, it is enough to require that the sandpile
\eqn{e:etastar}
{ \eta^* (x)
  = \begin{cases}
    \xi(x)       & \text{if $x \in W$;}\\
    \deg_G(x)-1  & \text{if $x \in V \setminus W$;}
    \end{cases} }
is in $\cR_G$ and for all $w \in W$ we have 
$\eta^* - \delta_w \not\in \cR_G$. This implies that 
$\xi$ is minimal if and only if $\xi \in \cR_{G_W}$, but
for any $w \in W$ we have $\xi - \delta_w \not\in \cR_{G_W}$.
\end{remark}

\begin{definition}
\label{defn:gen-minimal}
When the restriction that $G \setminus W$ be connected is
dropped, and $\xi$ satisfies the requirements of
Definition \ref{defn:minimal}, we say that $\xi$ is
\emph{generalized minimal}. (That is, in this case
we allow $W$ to have ``holes'').
\end{definition}

We extend Definition \ref{defn:minimal} 
to infinite graphs as follows.

\begin{definition}
\label{defn:minimal-infinite}
Let $G = (V,E)$ be a locally finite, connected, infinite graph, and 
let $W \subset V$ be a finite set such that all 
connected components of $G \setminus W$ are infinite.
Let $\xi$ be a stable configuration on $W$. 
We say that $\xi$ is \emph{minimal}, if for some (and 
then for any) finite $V_1 \supset W$ the configuration 
$\xi$ is minimal in the graph $G_{V_1}$.
When $G \setminus W$ is allowed to have finite components,
we call $\xi$ \emph{generalized minimal}, if it is 
generalized minimal with respect to some (and then for any) 
finite $V_1 \supset W$ for which $G \setminus V_1$ has only 
infinite components. 
\end{definition}

Theorems \ref{thm:minimal-finite} and \ref{thm:minimal-infinite}
below summarizes what can be proved using the method of 
Majumdar and Dhar \cite{MD91}. The extension to 
generalized minimal configurations appear to be new.
Let $\Delta_G$ denote the graph Laplacian on $G$, that is:
\eqnst
{ \Delta_G(x,y)
  = \begin{cases}
    \deg_G(x) & \text{if $x = y$;}\\
    - b_{xy} & \text{if $x \sim y$;}\\
    0       & \text{otherwise;}
    \end{cases} \qquad x, y \in V; }
where $x \sim y$ denotes that $x$ and $y$ are adjacent. 

\begin{theorem}
\label{thm:minimal-finite}
Let $G = (V \cup \{ s \}, E)$ be a finite, connected multigraph,
and let $\xi$ be minimal on $W \subset V$. There exists 
a subset $\cE$ of the set of edges touching $W$ such that 
\eqn{e:minimal}
{ \nu_G [ \eta : \eta_W = \xi ]
  = \det ( K_G(e,f) )_{e,f \in \cE}. }
The statement remains true for generalized minimal
configurations.
\end{theorem}

\begin{remark}
Alternatively, following the arguments of \cite{MD91},
the matrix can be replaced by some 
$R_{G,W}$ whose entries can be expressed in terms of 
$\Delta_G^{-1}(x,y)$, $x, y \in W \cup \partial_\extr W$, 
with 
$\partial_\extr W = \{ y \in V \setminus W : 
\text{$\exists$ $x \in W$, $x \sim y$} \}$.
\end{remark}

\begin{theorem}
\label{thm:minimal-infinite}
Let $G = (V,E)$ be a locally finite, connected, infinite graph, 
and $V_1 \subset V_2 \subset \dots \subset V$ any 
exhaustion. Let $W \subset V$ be finite, and let $\xi$ be minimal 
on $W$. There exists a subset $\cE$ of the set of edges touching 
$W$ such that 
\eqn{e:minimal2}
{ \lim_{n \to \infty} \nu_{V_n} [ \eta : \eta_W = \xi ]
  = \det ( K_G(e,f) )_{e,f \in \cE}. }
The statement remains true for generalized minimal
configurations.
\end{theorem}

\begin{remark}
When $G = \Z^d$, $d \ge 2$, the matrix can be replaced by 
some $R_W$ with entries expressed in terms of $a(x)$ or $G(x)$.
\end{remark}

We now state our general construction of sandpile measures.

\begin{theorem}
\label{thm:sandpile}
Let $G = (V,E)$ be a locally finite, connected, infinite graph. 
Suppose that $G$ satisfies the one-end property \eqref{e:one-end}.
There exists a unique measure $\nu$ on the space 
$\prod_{x \in V} \{ 0, \dots, \deg_G(x) - 1 \}$ such that along 
any exhaustion $V_1 \subset V_2 \subset \dots \subset V$
the measures $\nu_{V_n}$ converge weakly to $\nu$.
\end{theorem}

Our proof of Theorem \ref{thm:sandpile} is much simpler than 
earlier proofs in \cite{AJ04,Jarai11} in more restrictive
settings. We note however, that unlike the proof
of Theorem \ref{thm:sandpile}, the earlier proofs do 
give more than weak convergence, as 
they construct the sandpile measure as the image, under
a measurable map, of the WSF with extra randomness.

The proof of Theorem \ref{thm:sandpile} exhibits the 
limiting probability of a cylinder event as an infinite series. 
In the case of the event $\{ \eta : \eta(o) = k \}$ with 
some $k = 0, \dots, \deg_G(o)-1$, the series is a 
generalized version of \cite[Eqn.~(14)]{MD91}. The 
decomposition of the event into this series is
also implicit in \cite[Section 3]{Pr94}.

\section{Preliminaries}
\label{sec:prelims}

\subsection{The burning bijection}
\label{ssec:bijection}

In this section $G = (V \cup \{ s \}, E)$ is a finite,
connected multigraph
with sink $s$. Fix $Q \subset V$. We consider a particular version 
of the burning bijection \cite{MD92} depending on $Q$, for 
sandpiles on $V$. For the reader familiar with the usual construction,
we note that the idea is to burn in two phases: first we burn 
all vertices we can without touching the set $Q$, then 
the remaining vertices.

Fix a stable sandpile $\eta$ on $V$, and let 
\eqnst
{ B^{(1)}_{Q,0}
  = \{ s \} \qquad\qquad \text{and} \qquad\qquad
  U^{(1)}_{Q,0}
  = V, }
and for $i \ge 1$ define inductively
\eqnsplst
{ B^{(1)}_{Q,i}
  &= \left\{ x \in U^{(1)}_{Q,i-1} \setminus Q : 
     \eta(x) \ge \deg_{U^{(1)}_{Q,i-1}}(x) \right\} \\
  U^{(1)}_{Q,i}
  &= U^{(1)}_{Q,i-1} \setminus B^{(1)}_{Q,i}. }
We call $B^{(1)}_{Q,i}$ the vertices that burn at time $i$ in the
first phase, and $U^{(1)}_{Q,i}$ the vertices that remained unburnt
at time $i$. There exists a smallest index $I$ such that 
for $i \ge I$ we have $U^{(1)}_{Q,i+1} = U^{(1)}_{Q,i}$. To define
the second phase, set 
\eqnst
{ B^{(2)}_{Q,0}
  = \bigcup_{0 \le i \le I} B^{(1)}_{Q,i}
  = V \setminus U^{(1)}_{Q,I}  \qquad\qquad \text{and} \qquad\qquad
  U^{(2)}_{Q,0}
  = U^{(1)}_{Q,I}. }
Then for $i \ge 1$ we set
\eqnsplst
{ B^{(2)}_{Q,i} 
  &= \left\{ x \in U^{(2)}_{Q,i-1} : 
     \eta(x) \ge \deg_{U^{(2)}_{Q,i-1}}(x) \right\} \\
  U^{(2)}_{Q,i}
  &= U^{(2)}_{Q,i-1} \setminus B^{(2)}_{Q,i}. }
It is not difficult to show using \eqref{e:ample} that 
$U^{(2)}_{Q,i} = \es$ eventually if and only if $\eta \in \cR_V$.

We now define the burning bijection corresponding to the 
above burning rule. Fix for each 
$x \in V$ an ordering $<_x$ of the edges incident 
with $x$ in the graph $G_V$. We assign to 
$\eta \in \cR_V$ a spanning tree $t$ of $G_V$, by 
specifying for each $x \in V$ an oriented edge $e_x$
with $\underline{e_x} = x$ to be an edge of $t$ 
(here we write $\underline{e}$ for the tail of $e$ and 
$\overline{e}$ for the head of $e$). 
If $x \in B^{(1)}_{Q,i}$ for $i \ge 1$, then let 
\eqnsplst
{ n_x 
  &= \sum_{y \in \cup_{0 \le j < i} B^{(1)}_{Q,j}} b_{xy} \\
  P_x 
  &= \left\{ e : \underline{e} = x,\, 
     \overline{e} \in B^{(1)}_{Q,i-1} \right\}
  =: \{ e_0 <_x \dots <_x e_{|P_x|-1} \}. }
It follows from the burning rules that 
$\eta(x) = 2d - n_x + i$ for some $0 \le i < |P_x|$,
and hence we can define $e_x$ to be the $i$-th element
of $P_x$ in the order $<_x$. 
If $x \in B^{(2)}_{Q,i}$ for some $i \ge 1$, we make exactly
the same definitions replacing the superscript $(1)$ with 
$(2)$. It follows easily that $t$ is a spanning tree
of $G$. Note that we defined $t$ using directed edges,
and this way each edge of $t$ is directed towards $s$.
It is somewhat tedious, but fairly straightforward to 
check that the map is a bijection. We refer to this
as the \emph{``bijection based on $Q$''}.
When in the above $Q = \es$, we only have phase 1, and
we refer to this as the \emph{usual bijection}.
Observe that there is in fact much more flexibility 
in choosing the ordering $<_x$ than we stated above.
For example, we can allow the choice of $<_x$ to
depend on the set of vertices burnt up to the time
when $x$ is burnt. We will freely make use of this
flexibility in the sequel.

A special role will be played by the event that 
all vertices in $V \setminus Q$ can be burnt in the
first phase, that is the event
\eqn{e:EVQ}
{ E_{V,Q}
  = \{ \eta \in \cR_V : U^{(1)}_{Q,I}(\eta) = Q \}. }
Under the bijection based on $Q$, this corresponds
to the event that in $t$ there is no directed edge
pointing from $V \setminus Q$ to $Q$. 

\subsection{Minimal subconfigurations}
\label{ssec:minimal}

\begin{lemma}
\label{lem:burn_alone}
Let $G = (V \cup \{ s \},E)$ be a finite, connected multigraph, 
$W \subset V$ and let $\xi$ be minimal on $W$. For any 
$\eta \in \cR_G$ such that $\eta_W = \xi$, we have 
$\eta \in E_{V,W}$, i.e.~there is a 
burning sequence that burns all of $V \setminus W$ 
before burning any vertex in $W$.
\end{lemma}

\begin{proof}
We argue by contradiction. Suppose that $\eta \in \cR_G$,
$W \subsetneq U \subset V$, all of $V \setminus U$ can be
burnt before burning any vertex of $U$, but no further
vertex of $U \setminus W$ can be burnt without touching $W$. 
Let $v_1, v_2, \dots$
be a possible continuation of the burning of $\eta$ in $U$.
In particular, $v_1 \in W$. Let $i \ge 1$ be the smallest 
index such that $v_1, v_2, \dots, v_i \in W$ and $v_i$ 
neighbours a vertex in $U \setminus W$. Such an index has
to exist, since there will be a first time when a vertex
of $U \setminus W$ becomes burnable. Consider now 
$\xi' = \xi - \delta_{v_i}$. The sequence $v_1, v_2, \dots, v_i$
is a burning sequence for $\xi'$ that removes the vertex $v_i$,
and it follows that $\xi' \in \cR_{G_W}$. Recall that 
$G \setminus W$ is connected, so it follows that
with $\eta^*$ defined as in \eqref{e:etastar}, we have 
$\eta^* - \delta_{v_i} \in \cR_G$. This contradicts the 
assumption that $\xi$ is minimal, and hence the statement 
of the Lemma follows.
\end{proof}

The next lemma gives a recursive characterization of
minimal sandpiles.

\begin{definition}
Let $W \subset V$ such that $G \setminus W$ is connected.
The \emph{entry points} of $\xi \in \cR_{G_W}$ are the vertices 
$E(\xi,W) = \{ w_1, \dots, w_k \} \subset W$ that are 
burnable for $\xi$ in $W$ at the first step of the burning
algorithm.
\end{definition}

\begin{lemma}
\label{lem:reduction}
Let $W \subset V$ such that $G \setminus W$ is connected. 
Suppose that $\xi \in \cR_{G_W}$ is minimal, 
with entry points $E(\xi,W) = \{ w_1, \dots, w_k \}$.\\
(i) If $1 \le i < j \le k$, then $w_i \not\sim w_j$.\\
(ii) For each $i = 1, \dots, k$ we have 
$\xi(w_i) = \sum_{v \in W : v \sim w_i} b_{v w_i}$.\\
(iii) The subconfiguration $\xi_{W \setminus E(\xi,W)}$ is minimal.\\ 
Conversely, if $\xi \in \cR_{G_W}$ satisfies (i)--(iii), then
it is minimal.
\end{lemma}

\begin{proof}
(i) Suppose we had $w_i \sim w_j$. Then there is a burning
sequence for $\xi$ starting with $w_1, w_2$, and hence
$\xi(w_2) \ge \sum_{v \in W : v \sim w_2,\, v \not= w_1} b_{v w_2}$.
Due to minimality, we must have 
$\xi(w_2) = \sum_{ v \in W : v \sim w_2,\, v \not= w_1} b_{v w_2}$.
But this contradicts the assumption that $w_2 \in E(\xi,W)$.

(ii) Since $w_i$ is burnable in $\xi$, we must have 
$\xi(w_i) \ge \sum_{ v \in W : v \sim w_i } b_{v w_i}$. Again, due to
minimality, we must have equality here.

(iii) Burning all vertices in $E(\xi,W)$ gives a
configuration in $\cR_{W \setminus E(\xi,W)}$. This configuration
also has to be minimal, as otherwise $\xi$ would not be 
minimal.

Suppose now that $\xi \in \cR_W$ satisfies (i)--(iii). 
Due to (iii), for any $w \in W \setminus E(\xi,W)$ the configuration 
$\xi - \delta_w$ is not ample for some subset of $W \setminus E(\xi,W)$.
Let now $w \in E(\xi,W)$ and consider $\xi' = \xi - \delta_w$.
Write $E(\xi,W) = \{ w_1, \dots, w_k \}$, and assume the
indexing is such that $w = w_k$. In order to arrive at a 
contradiction assume that $\xi' \in \cR_W$. Note that $w = w_k$ 
is not burnable in $\xi'$ at the first step of the burning
algorithm, while the vertices $w_1, \dots, w_{k-1}$ 
are all still burnable in $\xi'$ at the first step. 
Hence there exists a burning sequence for $\xi'$ of the form 
\eqnst
{ v_1 = w_1, \dots, v_{k-1} = w_{k-1}, v_k, \dots, v_l = w_k, \dots, v_K }
where $W \setminus E(\xi,W) = \{ v_k, \dots, v_K \} \setminus \{ v_l \}$. 
Let $i \ge 1$ be the first index such that $v_i \sim w_k$. Note that 
necessarily $i < l$, and due to (i), $i \ge k$. Since at the time 
of burning of $v_i$ the vertex $w_k$ has not been burnt yet,
the sequence
\eqnst
{ v_k, \dots, v_i, \dots, v_{l-1}, v_{l+1}, \dots, v_K }
is a burning sequence for $\xi_{W \setminus E(\xi,W)} - \delta_{v_i}$.
This contradicts assumption (iii), and hence the proof is
complete.
\end{proof}

\begin{remark}
As a corollary, we obtain by induction the well known fact that 
all minimal configurations on $W$ have the same total number of 
particles and this equals the number of edges of $G_W$ minus the
degree of $s$ in $G_W$; see \cite[Theorem 3.5]{ML97} 
and \cite{Dhar06}.
\end{remark}

We next describe a burning procedure we can apply to 
generalized minimal configurations. Let $W \subset V$,
suppose that $\xi$ is generalized minimal on $W$, and
$\eta \in \cR_G$ such that $\eta_W = \xi$. Let 
$V_1, \dots, V_K$ be the connected components of
$G \setminus W$ not containing $s$. Let 
$W' := W \cup V_1 \cup \dots \cup V_K$. The burning
of $\eta$ is defined in several phases.

{\bf Phase 1.} By the same argument as Lemma 
\ref{lem:burn_alone} we get that $E_{V,W'}$ occurs,
so in Phase 1 we burn all of $V \setminus W'$.

{\bf Phase 2.} Burn vertices in $W$ as in 
usual burning, until the first time that a vertex
neighbouring some $V_j$ becomes burnable. 
Let $y^{(j)}_1, \dots, y^{(j)}_{r_j}$ be the 
vertices in $W$ neighbouring $V_j$ that became
burnable at this stage.

\begin{lemma}
\label{lem:burn_hole}
If $r_j \ge 1$, then after burning $y^{(j)}_1$, 
all vertices in $V_j$ can be burnt, without 
touching $W$.
\end{lemma}

\begin{lemma}
\label{lem:one_entry}
For each $1 \le j \le K$ we have $r_j = 0$ or $1$.
\end{lemma}

We prove these lemmas after we completed the definition
of the burning process.

{\bf Phase 3.} For each $j$ such that $r_j = 1$, burn 
$y^{(j)}_1$ and then burn all of $V_j$, appealing to
Lemmas \ref{lem:burn_hole} and \ref{lem:one_entry}.
Without loss of generality we assume that the $V_j$'s 
that were \emph{not} burnt are $V_1, \dots, V_{K_1}$ 
for some $0 \le K_1 < K$.

Following this we iterate Phases 2 and 3 for the
remaining vertices.

We use the above process to define an auxiliary graph
$G_W^*$. This is obtained from $G_{W'}$ by identifying
all vertices in $V_j$ with $y^{(j)}_1$, $1 \le j \le K$, 
and removing loop-edges. Then $\xi$ viewed as a 
configuration on $G_W^*$ is recurrent and minimal. 

Let us now prove the two lemmas.

\begin{proof}[Proof of Lemma \ref{lem:burn_hole}.]
This is similar to the proof of Lemma \ref{lem:burn_alone}.
Suppose there is a non-empty subset $U \subset V_j$
such that all of $V_j \setminus U$ can be burnt 
without touching $W$, but no further vertex
of $U$ can be burnt. Let $v_1, v_2, \dots$ be a 
continuation of the burning of $\eta_{W'}$. There 
is a first index $i$ such that $v_i$ neighbours a
vertex in $U$. Then we can apply the same sequence
to $\eta^* - \delta_{v_i}$, and see that 
this is in $\cR_G$. Since $v_i \in W$, this 
contradicts that $\xi$ was generalized minimal.
\end{proof}

\begin{proof}[Proof of Lemma \ref{lem:one_entry}.]
This is similar to the proof of 
Lemma \ref{lem:reduction}(i).
Suppose we have $r_j \ge 2$ for some $j$. Burn
$y^{(j)}_1$, and then all of $V_j$, using 
Lemma \ref{lem:burn_hole}. Since
$y^{(j)}_2$ neigbours $V_j$, this shows that 
we can decrease the number of particles at $y^{(j)}_2$,
contradicting the minimality of $\xi$.
\end{proof}

We now prove Theorem \ref{thm:minimal-finite}.

\begin{proof}[Proof of Theorem \ref{thm:minimal-finite}.]
By Lemma \ref{lem:burn_alone}, the event $\eta_W = \xi$ 
implies the event $E_{V,W}$. Using the burning rule
based on $W$, the conditional distribution of $\eta_W$ 
given the event $E_{V,W}$ is uniform on $\cR_W$. Hence
\eqn{e:min1}
{ \nu_G [ \eta : \eta_W = \xi ] 
  = \nu_G [ E_{V,W} ] (\det (\Delta_{G_W}))^{-1}. }
In order to write this in terms of the transfer 
current matrix, choose any spanning tree $t_0$ 
of $G_W$ (e.g.~the tree corresponding to $\xi$
under the burning bijection in the graph $G_W$), 
and let 
\eqnst
{ \cE
  := \{ \{ x, y \} : x \in W,\, \{ x, y \} \not\in t_0 \}. }
Then letting $t$ denote the tree corresponding to $\eta$
under the burning bijection based on $W$ in the graph $G$,
the event $\{ \cE \cap t = \es \}$ is equivalent to the event 
that $E_{V,W}$ occurs, and $\eta_W$ is a fixed
element of $\cR_W$. Hence by \eqref{e:extendedTCT},
the right hand side of \eqref{e:min1} equals
\eqnst
{ \det ( K_G(e,f) )_{e,f \in \cE}. }

Assume now that $\xi$ is generalized minimal. 
Consider the auxiliary graph $G_W^*$ constructed 
earlier. Let $t_0$ be the spanning tree assigned to 
$\xi$ under the bijection in the graph $G_W^*$. 
Let $\cE^*$ be the set of edges of $G_W^*$ not 
present in $t_0$. To each edge of $\cE^*$ 
corresponds an edge of $G$ touching $W$, let 
us call these edges $\cE$. Then the identification
of the burning processes on $G$ and $G_W^*$
shows that $\{ \eta_W = \xi \}$ is equivalent 
to $\{ \cE \cap T_G = \es \}$, where $T_G$ is
the Uniform Spannign Tree on $G$. Hence the 
statement follows with the set $\cE$.
\end{proof}

\begin{proof}[Proof of Theorem \ref{thm:minimal-infinite}.]
This is immediate from the proof of 
Theorem \ref{thm:minimal-finite},
as the event $T_{G_n} \cap \cE = \es$ is a cylinder event,
and hence its probability converges, as $n \to \infty$,
to $\mu [ T_G \cap \cE = \es ]$.
\end{proof}

\section{A general construction of sandpile measures}
\label{sec:sandpile}

In this section we prove Theorem \ref{thm:sandpile}. If $t$
is a spanning tree of $G_{V_n}$, and $x, y \in V_n$, we 
say that $x$ is a \emph{descendent} of $y$, if $y$ lies on 
the unique directed path from $x$ to $s$ (allowing 
$x = y$ as well).

\begin{proof}[Proof of Theorem \ref{thm:sandpile}.]
We need to show that for any fixed finite set $Q \subset V$
and $\xi \in \cR_Q$ the probabilities $\nu_{V_n} [ \eta_Q = \xi ]$
have a limit as $n \to \infty$. Without loss of generality 
assume that all components of $G \setminus Q$ are infinite.
Let $\eta \in \cR_{V_n}$ and let 
$W_0(\eta) \supset Q$ be the set of vertices that do not 
burn in the first phase, under the bijection based on $Q$. 

Fix $Q \subset W \subset V_n$, and assume the event $W_0(\eta) = W$.
We define an auxiliary graph $G_W^* = (W \cup \{ s \}, E_W^*)$ 
as follows. All edges of $G_{V_n}$ between vertices
$x, y \in W$ are also present in $E_W^*$. For every edge
$\{ x, y \} \in E_{V_n}$ satisfying $x \in Q$, 
$y \in (V_n \cup \{ s \}) \setminus W$ we place an edge 
between $x$ and $s$ in $E_W^*$. There are no other edges 
in $E_W^*$.

We claim that
\eqn{e:equiv}
{ W_0(\eta) = W,\, \eta_Q = \xi \qquad \text{if and only if} \qquad
  \eta \in E_{V_n,W},\, \eta_W \in \cR_{G_W^*},\, \eta_Q = \xi. }
To see this, assume first the left hand statement.
It is clear that $W_0(\eta) = W$ implies $E_{V_n,W}$, hence
we only need to prove that $\eta_W \in \cR_{G_W^*}$. 
It is clear that
\eqnst
{ \eta(x) < \deg_G(x) = \deg_{G_W^*}(x) \quad \text{for $x \in Q$.} }
Also, using that a vertex $x \in W \setminus Q$ does not burn
in phase one, we have 
\eqnst
{ \eta(x) < \deg_W(x) = \deg_{G_W^*}(x) \quad 
     \text{for $x \in W \setminus Q$.} }
Therefore, $\eta_W$ is a stable sandpile on $G_W^*$. 
Now it follows easily from the fact that $\eta_W$ burns 
in phase two that $\eta_W$ also burns in the graph
$G_W^*$ and hence $\eta_W \in \cR_{G_W^*}$.

Now assume the right hand statement in \eqref{e:equiv}.
Then we know that all vertices in $V_n \setminus W$ can 
be burnt without touching $W$, and in particular, 
without touching $Q$, so $W_0(\eta) \subset W$.
However, $\eta_W \in \cR_{G_W^*}$ implies that for
$x \in W \setminus Q$ we have $\eta(x) < \deg_W(x)$,
and hence no more vertices can be burnt in the first
phase, implying that $W_0(\eta) = W$. This proves
\eqref{e:equiv}.

The equivalence \eqref{e:equiv} hence gives the following 
decomposition:
\eqnsplst
{ \nu_{V_n} [ \eta_Q = \xi ]
  &= \sum_{W : Q \subset W \subset V_n} 
    \nu_{V_n} [ W_0(\eta) = W,\, \eta_Q = \xi ] \\
  &= \sum_{W : Q \subset W \subset V_n}\,
    \nu_{V_n} [ E_{V_n,W} ]\, \nu_{G_W^*} [ \eta_Q = \xi ]. }
It follows from the observation made after \eqref{e:EVQ}, 
that the event $E_{V_n,W}$ is spanning-tree-local, that is, 
it only depends on the status of the edges touching $W$.
Hence for fixed $W$ we have
\eqnst
{ \nu_{V_n} [ E_{V_n,W} ] \stackrel{n \to \infty}\longrightarrow 
  \text{some limit $p_W$.} }
Hence we get 
\eqnst
{ \lim_{n \to \infty} \nu_{V_n} [ \eta_Q = \xi ]
  = \sum_{\substack{W : Q \subset W \subset V\\\text{$W$ finite}}}
    p_W \nu_{G_W^*} [ \eta_Q = \xi ], }
provided we can show that for any $\eps > 0$ there exists
a finite $V_0 \subset V$ such that
\eqn{e:tail}
{ \sup_{n \ge 1} \nu_{V_n} [ W_0 \not\subset V_0 ] < \eps. } 

In order to show \eqref{e:tail}, we observe that under the
bijection based on $Q$, for every $\eta \in \cR_{V_n}$, 
$W_0(\eta)$ contains precisely those vertices that are 
descendents of some vertex of $Q$ in $t = t(\eta)$. 
Recall that due to the assumed one end property of the 
WSF on $G$, the notion of 
``descendent'' extends to the infinite case: 
$x$ is a descendent of
$y$ if and only if $y$ lies on the unique self-avoiding
path from $x$ to infinity. Let us write $\cD(Q)$ for the
set of descendents of $Q$. We have
$\mu [ |\cD(Q)| < \infty ] = 1$, hence there exists a 
finite $V'_0 \subset V$ such that 
$\mu [ \cD(Q) \not\subset V'_0 ] < \eps$.
Since $\mu_{V_n}$ converges weakly to $\mu$ and for fixed 
$W$ the event $\cD(Q) = W$ is a cylinder event,
we have $\mu_{V_n} [ \cD(Q) \not\subset V'_0 ] < \eps$
for all large enough $n$. Taking $V_0$ suitably larger 
than $V'_0$ we get $\mu_{V_n} [ \cD(Q) \not\subset V_0 ] < \eps$
for all $n \ge 1$. This proves \eqref{e:tail}, and 
completes the proof of the theorem.
\end{proof}

\end{document}